\documentclass[10pt,letterpaper]{article}
\usepackage{amsmath,amssymb,amsthm,mathtools}
\usepackage{enumitem,microtype,needspace}
\usepackage[hidelinks]{hyperref}

\setlist[enumerate,1]{label=(\roman*),leftmargin=25pt,itemsep=3pt,topsep=5pt,parsep=0pt}
\newtheorem{theorem}{Theorem}[section]
\newtheorem{lemma}[theorem]{Lemma}
\newtheorem{corollary}[theorem]{Corollary}

\theoremstyle{definition}
\newtheorem{definition}[theorem]{Definition}
\theoremstyle{remark}

\newcommand{\Int}{\operatorname{Int}}
\newcommand{\eps}{\varepsilon}

\newcommand{\dcup}{\mathbin{\dot\cup}}
\hypersetup{
  pdftitle={A linear bound for nested cycles without geometric crossings},
  pdfauthor={Jiangdong Ai, Gregory Gutin, Yiming Hao},
  pdfsubject={Nested cycles without geometric crossings},
  pdfkeywords={nested cycles, geometric crossings, extremal graph theory, sublinear expanders, graph minors}
}

\title{A linear bound for nested cycles without geometric crossings}
\author{Jiangdong Ai\thanks{School of Mathematical Sciences and LPMC, Nankai University, Tianjin 300071, P.R. China. \texttt{jd@nankai.edu.cn}. Partially supported by the National Natural Science Foundation of China (No. 12522117) and Fundamental and Interdisciplinary Disciplines Breakthrough Plan of the Ministry of Education of China (JYB2025XDXM207).}
\quad Gregory Gutin\thanks{Corresponding author. Department of Computing, Security and Mathematics, Royal Holloway, University of London, Egham, Surrey TW20 0EX, UK. Email: \texttt{g.gutin@rhul.ac.uk}.}
\quad Yiming Hao\thanks{School of Mathematical Sciences and LPMC, Nankai University, Tianjin 300071, P.R. China. Email: \texttt{1120230031@mail.nankai.edu.cn}.}}
\date{}

\begin{document}
\maketitle

\begin{abstract}
Cycles $C_1,\ldots,C_k$ in a graph are called nested without geometric crossings if they are pairwise edge-disjoint, $V(C_k)\subseteq\cdots\subseteq V(C_1)$, and each pair of consecutive cycles induces the same cyclic order on the vertices of the inner cycle, up to reversal. Let $f_k(n)$ be the least number of edges that forces such a family in every $n$-vertex graph. Gil Fern\'andez, Kim, Kim and Liu proved that $f_2(n)=O(n)$, answering a question of Erd\H{o}s, and asked whether $f_k(n)=O_k(n)$ for every fixed $k$. We prove this for all $k$. The proof selects the inner cycles together with a disjoint subgraph that supplies their external neighbours. A reselection argument gives disjoint paths from every inner-cycle vertex to any sufficiently large target set. Sublinear expansion and a rooted clique minor then allow the vertices to be joined in the required cyclic order.
\end{abstract}



\section{Introduction}\label{sec:introduction}

All graphs are finite, simple and undirected. For a finite set $U$, $|U|$ denotes its cardinality. For a graph $G$, let $V(G)$ and $E(G)$ denote its vertex and edge sets, respectively, and write $|G|=|V(G)|$ and $e(G)=|E(G)|$. An edge with endpoints $u,v$ is written $uv$. For $U\subseteq V(G)$, the induced subgraph $G[U]$ has vertex set $U$ and all edges of $G$ with both endpoints in $U$. For a cycle $C$, its length is $|C|=|V(C)|=|E(C)|$.

If $C$ is a cycle and $X\subseteq V(C)$, the \emph{cyclic order} induced by $C$ on $X$ is the order in which the vertices of $X$ are encountered when traversing $C$, considered up to cyclic shift.

\begin{definition}\label{def:nested}
A sequence $C_1,\ldots,C_k$ of cycles in a graph is a family of $k$ \emph{nested cycles without geometric crossings} if
\begin{enumerate}
\item $V(C_k)\subseteq V(C_{k-1})\subseteq\cdots\subseteq V(C_1)$;
\item the cycles are pairwise edge-disjoint; and
\item for every $i<k$, the cyclic order induced by $C_i$ on $V(C_{i+1})$ agrees, up to reversal, with the cyclic order of $C_{i+1}$.
\end{enumerate}
\end{definition}

We call such a sequence a \emph{$k$-layer family}; $C_1$ is its \emph{outermost cycle} and $V(C_1)$ its \emph{outer vertex set}.

For integers $k\ge2$ and $n\ge1$, let $f_k(n)$ be the least integer such that every $n$-vertex graph with at least $f_k(n)$ edges contains such a family. When $C_i$ is drawn as a convex polygon, conditions (ii) and (iii) say that the edges of $C_{i+1}$ are noncrossing chords. A family with $k=1$ consists simply of a cycle. We use this case as the base of an induction.

In 1975, Erd\H{o}s~\cite{Erdos} asked several questions about edge-disjoint cycles with additional structure, including nested cycles without geometric crossings. If condition (iii) is omitted, Bollob\'as~\cite{Bollobas} proved that a linear number of edges forces two nested cycles, and Chen, Erd\H{o}s and Staton~\cite{CES} extended this to every fixed number of cycles. These results do not impose a cyclic order on the common vertices. Gil Fern\'andez, Kim, Kim and Liu~\cite{GKKL} proved $f_2(n)=O(n)$ and asked whether the same conclusion holds for every fixed $k$.

Xu, Zeng and Zhang~\cite{XZZ} obtained the first general upper bound,
\[
f_k(n)=O_k\!\left(n(\log n)^{k-1}(\log\log n)^{k-3}\right)
\qquad(k\ge3).
\]
Their proof builds the layers successively in a robust sublinear expander, controlling the length of the outer cycle at each step. In~\cite{AGH}, we proved
\[
f_k(n)=O_k\!\left(n\,\frac{(\log\log n)^2}{\log\log\log n}\right)
\qquad(k\ge3).
\]
That argument derives expansion and local sparsity from a slightly superlinear extremal threshold. The local sparsity permits the construction of private trees at the vertices of the inner cycle. The density comparison used to grow these trees, however, does not hold with a constant weight.

Our main result removes the remaining dependence on $n$ from the density threshold and answers the question of Gil Fern\'andez, Kim, Kim and Liu~\cite[Question 1.2]{GKKL}.

\begin{theorem}\label{thm:main}
For every fixed integer $k\ge1$, there is a constant $C_k>0$ such that every graph $G$ with $e(G)>C_k|G|$ contains $k$ nested cycles without geometric crossings. Consequently, $f_k(n)=O_k(n)$ for every fixed $k\ge2$.
\end{theorem}

The order of the bound in $n$ is best possible: an $n$-vertex unicyclic graph (a connected graph with exactly one cycle) has $n$ edges and does not contain two edge-disjoint cycles. We make no attempt to optimise the dependence of $C_k$ on $k$.

The sublinear expander method originates in work of Koml\'os and Szemer\'edi~\cite{KS}; robust forms were developed by Haslegrave, Kim and Liu~\cite{HKL} and by Alon, Buci\'c, Sauermann, Zakharov and Zamir~\cite{ABSZZ}. We need only ordinary vertex expansion, and give a direct extraction proof. The final step uses the separator theorem of Alon, Seymour and Thomas~\cite{AST} and a rooted clique-minor lemma in the form given by Nguyen~\cite{Nguyen}. The latter is related to the disjoint-paths theorem of Robertson and Seymour~\cite{RS}; a linkage formulation is also stated by Kawarabayashi, Kobayashi and Reed~\cite{KKR}.

\subsection*{Overview of the proof}

Assume that a constant density forces $j$ layers. In a weak expander $H$ of sufficiently large minimum degree, we choose a $j$-layer family together with a disjoint \emph{reservoir} $B$. If $X$ is the vertex set of its outer cycle, every vertex of $X$ is required to have at least $2r$ neighbours in $B$, and $H[B]$ has minimum degree at least $r$, where $r$ depends only on the preceding density constant. Among all families admitting such a reservoir, choose one with $q=|X|$ minimum. Both the family and its reservoir may change in this minimisation.

This choice first gives a local edge bound on all sets of fewer than $q$ vertices. Section~\ref{sec:local} proves a short-cycle construction for graphs satisfying a local edge bound at a fixed polylogarithmic scale. Applied to the selected family, it shows that $q$ is polylogarithmic in $|H|$. This is a consequence of the selection, not an additional induction hypothesis.

The main reselection argument is in Section~\ref{sec:selection}. We first show how to keep a prescribed external subgraph of minimum degree at least $r$ while choosing a dense core on the remaining vertices. This gives a second local edge bound, with a constant depending on the preceding density threshold rather than on $r$. The two bounds rule out all deficient cuts in a vertex-capacitated network. As a result, every vertex of $X$ has two paths to any sufficiently large target set in the reservoir, and the paths are disjoint outside $X$.

Section~\ref{sec:proof} uses expansion and the separator theorem to find a sufficiently large clique minor after deleting $X$ and fewer than $q$ additional vertices. Replace each root by two clones. The path property excludes every small separation between the clones and this minor, so the rooted clique-minor lemma applies. Pairing the clones of consecutive roots gives the additional outer cycle. This cycle need not be short. The induction uses only the existence of the preceding layers and chooses a new short family at the next step.

\subsection*{Notation}

For $v\in V(G)$, let $N_G(v)=\{w\in V(G):vw\in E(G)\}$ be its neighbourhood and $d_G(v)=|N_G(v)|$ its degree. For a nonempty graph $G$, write $\delta(G)=\min_{v\in V(G)}d_G(v)$, $\Delta(G)=\max_{v\in V(G)}d_G(v)$ and $d(G)=2e(G)/|G|$ for its minimum, maximum and average degrees. We use $\Delta(G)=0$ when $V(G)=\varnothing$. For $U\subseteq V(G)$, we distinguish
\[
\begin{aligned}
N_G(U)&=\{v\in V(G)\setminus U:uv\in E(G)\text{ for some }u\in U\},\\
\Gamma_G(U)&=\bigcup_{u\in U}N_G(u),\qquad N_G(U)=\Gamma_G(U)\setminus U.
\end{aligned}
\]
Thus $N_G(U)$ is the external neighbourhood and never meets $U$, whereas $\Gamma_G(U)$ may meet $U$. In particular, $N_G(v)=N_G(\{v\})$. For $v\in V(G)$ and $W\subseteq V(G)$, let $d_G(v,W)=|N_G(v)\cap W|$; for disjoint $U,W\subseteq V(G)$, let $e_G(U,W)=|\{uv\in E(G):u\in U,\ v\in W\}|$. All subscripts refer to the graph in which neighbours and degrees are counted. For example, if $U\subseteq R\subseteq V(H)$, then $N_{H[R]}(U)=(\Gamma_H(U)\cap R)\setminus U$, and $d_{H[R]}(v)=d_H(v,R)$ for $v\in R$.

We write $H\subseteq G$ for a subgraph, not necessarily induced. Vertex deletion means $G-U=G[V(G)\setminus U]$, and $G-v=G-\{v\}$. The union of subgraphs has the unions of their vertex sets and edge sets. The symbol $\dcup$ denotes a union of pairwise disjoint sets. Disjointness or avoidance between subgraphs refers to their vertex sets unless edge-disjointness is specified. A path has distinct vertices and its length is its number of edges; a path of length zero is a single vertex. If $P$ has endpoints $u,v$, then $\Int(P)=V(P)\setminus\{u,v\}$, and $P[a,b]$ is the subpath between $a,b\in V(P)$, including its endpoints. An $X$--$Y$ path has one endpoint in each set and no internal vertex in $X\cup Y$; a common vertex gives a path of length zero. Paths are \emph{disjoint outside $X$} if their sets $V(P)\setminus X$ are pairwise disjoint. They are \emph{internally vertex-disjoint} if their internal vertex sets are pairwise disjoint and no endpoint of one is internal to another. A ball of radius $t$ about a vertex set consists of the vertices reachable from that set by a path of length at most $t$.

We use $\mathbb N=\{1,2,\ldots\}$; $\lfloor x\rfloor$ and $\lceil x\rceil$ are the greatest integer at most $x$ and the least integer at least $x$, respectively. All logarithms are natural except $\log_2$. 

\section{Preliminaries}\label{sec:preliminaries}

We begin with two elementary facts used in the local construction.

\begin{lemma}\label{lem:prune}
Every graph $G$ with at least one edge has a nonempty subgraph $H$ with $\delta(H)\ge e(G)/|G|$. More generally, for any $a>0$, if $e(G)>a|G|$, repeatedly deleting vertices of current degree less than $a$ leaves a nonempty graph $H$ with $\delta(H)\ge a$ and $e(H)>a|H|$.
\end{lemma}
\begin{proof}
For the first statement, let $a=e(G)/|G|$ and repeatedly delete a vertex whose current degree is less than $a$. If every vertex were deleted, summing the degrees at deletion would give fewer than $a|G|=e(G)$ edges, although every edge is counted once. For the second statement, the strict inequality $e(H)>a|H|$ is preserved by every deletion and is false for the empty graph. Thus the deletion process cannot remove all vertices.
\end{proof}

\begin{lemma}\label{lem:shortcycle}
Every graph $H$ with $\delta(H)\ge3$ contains a cycle of length at most $2\log_2|H|+3$.
\end{lemma}
\begin{proof}
Let $m=|H|$ and $t=\lceil\log_2m\rceil$. If no cycle has length at most $2t+1$, every breadth-first search ball of radius $t$ is a tree: an edge outside the search tree would close a cycle of length at most $2t+1$. The root has at least three children and each nonroot vertex at depth less than $t$ has at least two. The ball therefore has at least $1+3(2^t-1)>m$ vertices, a contradiction. Thus the required cycle has length at most $2t+1\le2\log_2m+3$.
\end{proof}

\subsection{Expansion and reservoirs}

Throughout the paper, let
\[
\eps=\frac1{16},\qquad
\rho(x)=\frac{\eps}{\log^2(15x)}\quad(x\ge1),\qquad
R(n)=\left\lceil160\log^3(15n)\right\rceil\quad(n\ge1).
\]
Here $\rho$ is the expansion rate and $R(n)$ is a path-length bound for graphs of order at most $n$. We call $H$ a \emph{weak expander} if, for every $U\subseteq V(H)$,
\begin{equation}\label{eq:expand}
|N_H(U)|\ge\rho(|U|)|U|
\qquad(1\le|U|\le |H|/2).
\end{equation}
No robustness under edge deletion is included in this definition. The following extraction argument is a bounded-potential version of the minimal-graph argument in~\cite{AGH}.

\begin{lemma}\label{lem:extract}
Let $M>0$. If $e(G)>M|G|$, then $G$ has an induced subgraph $H$ which is a weak expander and satisfies $e(H)>M|H|/2$ and $\delta(H)>M/2$.
\end{lemma}
\begin{proof}
For $x\ge1$, let $\psi(x)=2-1/\log(15x)$ and $\Psi(x)=x\psi(x)$. Writing a prime for differentiation with respect to $x$, we obtain
\[
1<\psi(x)<2,\qquad
1<\Psi'(x)=2-\frac1{\log(15x)}+\frac1{\log^2(15x)}<2.
\]
Since $e(G)>(M/2)\Psi(|G|)$, choose an induced subgraph $H$ of minimum order $N$ satisfying this inequality. Necessarily $N\ge2$, and every nonempty proper $W\subset V(H)$ satisfies $e(H[W])\le M\Psi(|W|)/2$. The bound on $e(H)$ follows from $\psi>1$. For each $v\in V(H)$, minimality and $\Psi'>1$ give $d_H(v)>(M/2)\bigl(\Psi(N)-\Psi(N-1)\bigr)>M/2$.

Let $\varnothing\ne U\subseteq V(H)$, where $u=|U|\le N/2$, and write $S=N_H(U)$ and $s=|S|$. If $U\cup S=V(H)$, then $s=N-u\ge u$, and the required expansion holds. Otherwise both $U\cup S$ and $V(H)\setminus U$ are nonempty proper sets. No edge joins $U$ to $V(H)\setminus(U\cup S)$, so
\[
e(H)\le e(H[U\cup S])+e(H[V(H)\setminus U]).
\]
Together with the preceding bound on proper induced subgraphs, this yields
\[
\Psi(N)-\Psi(N-u)-\Psi(u)
<\Psi(u+s)-\Psi(u)\le2s.
\]
The left side is at least $u(\psi(N)-\psi(u))$, since $\psi$ is increasing. Also $N\ge2u$, and hence
\[
\psi(N)-\psi(u)\ge\frac{\log2}{\log(15u)\log(30u)}
\ge\frac1{4\log^2(15u)}.
\]
The last inequality uses $\log2>1/2$ and $\log(30u)\le2\log(15u)$. Thus the covering inequality gives $s>u/(8\log^2(15u))$, which implies~\eqref{eq:expand}.
\end{proof}

\begin{lemma}\label{lem:connect}
Let $H$ be a weak expander with $|H|\le n$, and let $Z\subseteq V(H)$. Suppose that $X,Y\subseteq V(H)\setminus Z$, $|X|,|Y|\ge x\ge1$, and $|Z|\le\rho(x)x/4$. Then $H-Z$ contains an $X$--$Y$ path of length at most $R(n)$.
\end{lemma}
\begin{proof}
Let $N=|H|$ and $b=\log(15N)$. The function $u\rho(u)$ is increasing for $u\ge1$. For $U\subseteq V(H)\setminus Z$ with $x\le|U|\le N/2$, therefore,
\[
|N_{H-Z}(U)|\ge\rho(|U|)|U|-|Z|
\ge\frac34\rho(|U|)|U|
\ge\frac{\eps}{2b^2}|U|.
\]
Let $\beta=\eps/(2b^2)$. Starting from $X$, balls in $H-Z$ grow by a factor at least $1+\beta$ until they exceed $N/2$ vertices. Since $0<\beta<1$ and $\log(1+\beta)\ge\beta/2$, a ball of radius $T=\left\lceil4\eps^{-1}b^3\right\rceil$ contains more than $N/2$ vertices. Otherwise its order would be at least $\exp(T\eps/(4b^2))\ge15N$. The same holds for the ball grown from $Y$. The two balls meet, giving a path of length at most
\[
2T\le8\eps^{-1}b^3+2\le10\eps^{-1}b^3\le R(n).
\]
A starting set of order greater than $N/2$ already meets the size requirement. If $X\cap Y\ne\varnothing$, a path of length zero suffices.
\end{proof}

\begin{lemma}\label{lem:reservoir}
If $\delta(H)\ge256$, then there are disjoint $A,B\subseteq V(H)$ such that
\[
e(H[A])\ge\frac{e(H)}{128},\qquad
 d_H(v,B)\ge
 \begin{cases}
 d_H(v)/4,&v\in A,\\
 d_H(v)/8,&v\in B.
 \end{cases}
\]
\end{lemma}
\begin{proof}
Include each vertex independently in $A_0$ with probability $1/8$, and let $B_0=V(H)\setminus A_0$. For $U\subseteq V(H)$, write $\mu(U)=\sum_{v\in U}d_H(v)$, and define $B^{\times}=\{v:d_H(v,A_0)>d_H(v)/4\}$. For an integer $d\ge0$, if $Z\sim\operatorname{Bin}(d,1/8)$, Markov's inequality gives
\[
\Pr(Z>d/4)\le2^{-d/4}\mathbb E2^Z
=2^{-d/4}(9/8)^d\le e^{-d/24},
\]
where $\log(9/8)\le1/8$ and $\log2\ge2/3$. Consequently, $\mathbb E\mu(B^{\times})\le2e(H)e^{-\delta(H)/24}$.

Starting from $B_0$, repeatedly delete a vertex whose degree in the remaining set is less than one eighth of its original degree in $H$. Let $L$ be the deleted set and $B=B_0\setminus L$. A deleted vertex outside $B^{\times}$ initially has at least $3d_H(v)/4$ neighbours in $B_0$ and fewer than $d_H(v)/8$ at deletion. More than $5d_H(v)/8$ of its neighbours have therefore been deleted earlier. Counting each edge of $H[L]$ at its later-deleted endpoint gives
\[
\frac58\mu(L\setminus B^{\times})\le e(H[L])\le\frac12\mu(L),
\qquad \mu(L)\le5\mu(B^{\times}).
\]
Let $A_{\rm bad}=\{v\in A_0:d_H(v,B)<d_H(v)/4\}$ and $A=A_0\setminus A_{\rm bad}$. Every vertex of $A_{\rm bad}\setminus B^{\times}$ has more than half of its original degree into $L$. Since $A_0\cap L=\varnothing$, $\mu(A_{\rm bad})\le\mu(B^{\times})+2\mu(L)\le11\mu(B^{\times})$. For every outcome, $e(H[A])\ge e(H[A_0])-\mu(A_{\rm bad})$. Hence the bound on $\mathbb E\mu(B^{\times})$ and $\mathbb E e(H[A_0])=e(H)/64$ imply
\[
\mathbb E e(H[A])\ge e(H)\left(\frac1{64}-22e^{-\delta(H)/24}\right)
\ge e(H)/128.
\]
The last inequality holds for $\delta(H)\ge256$. Choose an outcome with at least this many edges. The degree bounds follow from the deletion rule and the definition of $A$.
\end{proof}

\subsection{Clique minors}

For an integer $h\ge1$, $K_h$ denotes the complete graph on $h$ vertices. A \emph{minor} is obtained by deleting vertices or edges and contracting edges, discarding loops and duplicate edges. A \emph{$K_h$ model} in a graph $G$ is a collection of $h$ pairwise disjoint nonempty vertex sets, called its \emph{branch sets}, each inducing a connected subgraph of $G$, with an edge between every pair. A graph contains a $K_h$ minor if and only if it has such a model. A \emph{separation} $(A,B)$ of $G$ satisfies $A\cup B=V(G)$ and has no edge between $A\setminus B$ and $B\setminus A$. Its \emph{separator} is $A\cap B$, its order is $|A\cap B|$, and its two \emph{open sides} are $A\setminus B$ and $B\setminus A$.

We use the following form of the separator theorem of Alon, Seymour and Thomas~\cite[Theorem~(1.2)]{AST}.

\begin{lemma}\label{lem:AST}
If an $m$-vertex graph has no $K_h$ minor, its vertex set has a partition $A\dcup Z\dcup B$ with no edge between $A$ and $B$, where $|A|,|B|\le2m/3$ and $|Z|\le h^{3/2}\sqrt m$.
\end{lemma}

For $Y\subseteq V(J)$ with $|Y|=p$, a $K_p$ model is \emph{rooted at $Y$} if every branch set contains exactly one vertex of $Y$. We also use the following special case of Nguyen~\cite[Lemma~5.7]{Nguyen}.

\begin{lemma}\label{lem:rooted}
Let $Y\subseteq V(J)$ have size $p\ge1$, and suppose that $J-Y$ contains a $K_h$ model with $h\ge2p$. If there is no separation $(A,B)$ of $J$ of order less than $p$ such that $Y\subseteq A$ and a whole branch set of this model is contained in $B\setminus A$, then $J$ has a $K_p$ model rooted at $Y$.
\end{lemma}
\begin{proof}
Apply~\cite[Lemma~5.7]{Nguyen} with its parameters $m=h$, $t=p$ and $n=0$. The given branch sets are connected, disjoint from $Y$, and pairwise adjacent. Its separation hypothesis is exactly the one above. The conclusion is a $K_{h-p}$ model attached to $Y$, meaning that $p$ of its branch sets each contain exactly one vertex of $Y$ and together contain all of $Y$. Since $h-p\ge p$, retaining those $p$ branch sets gives the required rooted model.
\end{proof}

The separation condition in Lemma~\ref{lem:rooted} refers to a whole branch set on the opposite side, not merely to a vertex of that set. We verify this condition directly when applying the lemma. Related separation-to-linkage formulations appear in~\cite[Theorem~(5.4)]{RS} and~\cite[Theorem~4.1]{KKR}.

\section{A local construction}\label{sec:local}

This section proves the short-family result used when a small set is too dense. The graph under consideration may have fewer than $N$ vertices, but all local bounds and path lengths will be measured at the same fixed scale $N$.

\subsection{Private trees and an additional outer cycle}

We use the standard extendability formulation of Montgomery~\cite[Definition~3.1 and Lemma~3.3]{Montgomery}. We include the leaf-extension proof, which in~\cite{Montgomery} is attributed to Glebov, Johannsen and Krivelevich.

\begin{definition}\label{def:extend}
Let $D\ge3$ and $m\ge1$ be integers. A subgraph $S\subseteq J$ is \emph{$(D,m)$-extendable} if $\Delta(S)\le D$ and, for every $U\subseteq V(J)$ with $|U|\le2m$,
\begin{equation}\label{eq:extend}
|\Gamma_J(U)\setminus V(S)|
\ge(D-1)|U|-\sum_{v\in U\cap V(S)}(d_S(v)-1).
\end{equation}
\end{definition}

The subgraph $S$ in this definition need not be induced. The degrees in the sum are its degrees in $S$.

\begin{lemma}\label{lem:leaf}
Suppose that $S\subseteq J$ is $(D,m)$-extendable and that every $U\subseteq V(J)$ with $m\le|U|\le2m$ satisfies
\[
|\Gamma_J(U)|\ge |V(S)|+2Dm+1.
\]
If $v\in V(S)$ and $d_S(v)\le D-1$, then there is a neighbour $y\notin V(S)$ of $v$ such that adding the vertex $y$ and the edge $vy$ to $S$ preserves $(D,m)$-extendability.
\end{lemma}
\begin{proof}
For $U\subseteq V(J)$, define
\[
F_S(U)=|\Gamma_J(U)\setminus V(S)|-(D-1)|U|
+\sum_{u\in U\cap V(S)}(d_S(u)-1).
\]
The function $F_S$ is \emph{submodular}: $F_S(U)+F_S(W)\ge F_S(U\cup W)+F_S(U\cap W)$ for all $U,W\subseteq V(J)$. Indeed, the neighbourhood of a union is the union of the neighbourhoods, whereas the neighbourhood of an intersection is contained in their intersection. The other two terms are modular, meaning that they satisfy equality in this set-function inequality. Call $U$ \emph{tight} if $|U|\le2m$ and $F_S(U)=0$.

Every tight set has size less than $m$. Otherwise, the hypothesis and $\sum_{u\in U\cap V(S)}(d_S(u)-1)\ge-|U|$ would give $F_S(U)\ge2Dm+1-D|U|\ge1$. If $U,W$ are tight, then $|U\cup W|<2m$, and submodularity together with extendability gives
\[
0\le F_S(U\cup W)+F_S(U\cap W)\le F_S(U)+F_S(W)=0.
\]
Thus $U\cup W$ is tight and has size less than $m$. It follows by induction that the union of any finite collection of tight sets is tight and has size less than $m$.

Suppose that every $y\in N_J(v)\setminus V(S)$ is unsuitable. This set is nonempty, since~\eqref{eq:extend} applied to $\{v\}$ gives at least $D-d_S(v)\ge1$ choices. Write $S_y=S+vy$ for the subgraph with vertex set $V(S)\cup\{y\}$ and edge set $E(S)\cup\{vy\}$. For every $U\subseteq V(J)$,
\[
F_{S_y}(U)=F_S(U)+\mathbf1_{\{v\in U\}}-\mathbf1_{\{y\in\Gamma_J(U)\}}.
\]
The maximum-degree condition still holds. Since $F_S(U)$ is a nonnegative integer for $|U|\le2m$, failure of extendability for $S_y$ gives a tight set $U_y$ with $v\notin U_y$ and $y\in\Gamma_J(U_y)$.

Let $W=\bigcup_{y\in N_J(v)\setminus V(S)}U_y$. Then $v\notin W$, $|W|<m$, $F_S(W)=0$, and $N_J(v)\setminus V(S)\subseteq\Gamma_J(W)\setminus V(S)$. Adding $v$ to $W$ introduces no new neighbour outside $V(S)$. Hence
\[
F_S(W\cup\{v\})=F_S(W)-(D-1)+(d_S(v)-1)=d_S(v)-D<0.
\]
As $|W\cup\{v\}|\le m$, this contradicts extendability.
\end{proof}

In a rooted tree, a vertex's depth is the number of edges on its path from the root; its children are its neighbours at the next depth. A complete binary tree of depth $h$ has its root at depth zero, two children at every vertex of depth less than $h$, and $2^h$ leaves at depth $h$. In the next lemma the word \emph{private} means that the trees have disjoint vertex sets; extra edges between them in the ambient graph are allowed.

\begin{lemma}\label{lem:trees}
Let $H\subseteq G$, let $M>0$, and let $A,B\subseteq V(H)$ be disjoint. Suppose that $d_H(v,B)>M/16$ for every $v\in A\cup B$. Let $C\subseteq H[A]$ be a cycle of length $q$. For an integer $h\ge1$, let $s=2^h$ and $m=4qs$. If every $W\subseteq V(G)$ satisfies
\begin{equation}\label{eq:localsparse}
e(G[W])\le(M/1024)|W|\qquad(0<|W|\le26m),
\end{equation}
then every vertex of $C$ has two distinct neighbours in $B$, each of which is the root of a complete binary tree of depth $h$ in $H[B]$, and all $2q$ trees are pairwise vertex-disjoint.
\end{lemma}
\begin{proof}
Let $X=V(C)$ and $J=H[X\cup B]$. For nonempty $U\subseteq V(J)$ with $|U|\le2m$, write $Y=\Gamma_J(U)\cap B$. By the degree assumption, the sum of the degrees from $U$ into $B$ exceeds $M|U|/16$. Each edge is counted at most twice, so $e(G[U\cup Y])>M|U|/32$. If $|Y|<12|U|$, then $|U\cup Y|<13|U|\le26m$, and~\eqref{eq:localsparse} gives $e(G[U\cup Y])<13M|U|/1024<M|U|/32$, a contradiction. Therefore
\begin{equation}\label{eq:expandB}
|\Gamma_J(U)\cap B|\ge12|U|\qquad(0<|U|\le2m).
\end{equation}

Begin with the working subgraph $S=C$. Since $B\cap X=\varnothing$ and every vertex of $C$ has degree two in $S$, inequality~\eqref{eq:expandB} implies that $S$ is $(4,m)$-extendable. The completed subgraph will have $q+2q(2^{h+1}-1)=4qs-q<m$ vertices. Thus every intermediate subgraph has fewer than $m$ vertices. For $m\le|U|\le2m$, inequality~\eqref{eq:expandB} gives $|\Gamma_J(U)|\ge12m\ge |V(S)|+8m+1$, so Lemma~\ref{lem:leaf} remains applicable throughout the construction.

Add two root edges at every vertex of $C$, and then add the edges of the desired trees one leaf at a time. Before each addition, the parent has degree at most three in $S$: each cycle vertex finishes with degree four, and each tree vertex with degree at most three. Every new vertex is in $B$, because all vertices of $X$ already belong to $S$. This gives the required trees.
\end{proof}

When $C$ is the outermost cycle of an existing family, we use only its edges in the working subgraph. All vertices of the other cycles are already in $V(C)$, so no old vertex can be selected as a new tree vertex.

\begin{lemma}\label{lem:wrap}
Let $H$ be a weak expander with $|H|\le n$, and let $C=x_1\cdots x_qx_1$ be a cycle of length $q\ge3$, listed in its cyclic order. Suppose that each $x_i$ is joined to the roots of two complete binary trees $T_i^-,T_i^+$ of integer depth $h$, all trees pairwise vertex-disjoint and outside $V(C)$. The signs $-$ and $+$ distinguish the two trees at a vertex. Indices on $x_i$ and $T_i^\pm$ are read modulo $q$, so $x_{q+1}=x_1$ and $T_{q+1}^\pm=T_1^\pm$. Let $s=2^h$, $p=\lfloor h/2\rfloor$, and $R_0=R(n)$. If $h\ge2$ and
\begin{align}
q(R_0+1)&\le2^{p-1},\label{eq:B1}\\
q+4q2^p+q(R_0+2h+2)&\le\frac{\eps s}{8\log^2(15s)},\label{eq:B2}
\end{align}
then $H$ has a cycle $C^+$ containing $V(C)$ in the cyclic order of $C$, with no edge in $H[V(C)]$, and with $|C^+|\le q(R_0+2h+2)$.
\end{lemma}
\begin{proof}
Write $X=V(C)$. Protect all vertices at depths $0,\ldots,p$ in every private tree, and let $P^*$ be their union. Then $|P^*|=2q(2^{p+1}-1)<4q2^p$. We construct paths $S_i$ from $x_i$ to $x_{i+1}$, using $T_i^+$ and $T_{i+1}^-$ at step $i$. Each private tree is used at exactly one step. Each $S_i$ consists of two root edges, two tree paths of length at most $h$, and a \emph{core path} $P_i$ of length at most $R_0$. The core, including its endpoints, will avoid $X$, $P^*$ and the interiors of all previously chosen paths.

Suppose that $S_1,\ldots,S_{i-1}$ have been chosen, and let $Q_{i-1}=\bigcup_{1\le l<i}V(P_l)$ and $I_{i-1}=\bigcup_{1\le l<i}\Int(S_l)$. Empty unions are understood to be empty, so $Q_0=I_0=\varnothing$. We have $|Q_{i-1}|\le q(R_0+1)$ and $|I_{i-1}|\le q(R_0+2h+2)$. A leaf of an unused tree is \emph{live} if its root-to-leaf path avoids $Q_{i-1}$. Every vertex of $Q_{i-1}$ in that tree has depth at least $p+1$, because all previous cores avoid $P^*$. Such a vertex lies on at most $s/2^{p+1}$ root-to-leaf paths. By~\eqref{eq:B1}, each unused tree therefore has at least $s-|Q_{i-1}|s/2^{p+1}\ge3s/4$ live leaves.

In each of the two current trees, take the union of all root-to-live-leaf paths, and denote the resulting rooted subtrees by $F,F'$. They are disjoint and avoid $I_{i-1}$. Indeed, they avoid old cores by liveness, and the old tree paths lie in different private trees. Their live leaves lie outside $Z_i=X\cup P^*\cup I_{i-1}$. By~\eqref{eq:B2} and the bounds on $|P^*|$ and $|I_{i-1}|$,
\[
|Z_i|\le\frac{\eps s}{8\log^2(15s)}
\le\frac{\rho(s/2)(s/2)}4.
\]
Lemma~\ref{lem:connect} gives a path $Q$ of length at most $R_0$ between the two live-leaf sets in $H-Z_i$.

The path $Q$ is allowed to meet non-leaf vertices of the current trees. Traverse $Q$ from its endpoint in $F$. Let $b$ be the first vertex of $F'$ encountered, and let $a$ be the last vertex of $F$ encountered before $b$. Then $P_i=Q[a,b]$ has its interior disjoint from $F\cup F'$. Append the unique root-to-$a$ and root-to-$b$ paths in $F,F'$, together with the two root edges. The resulting $S_i$ is a simple $x_i$--$x_{i+1}$ path of length at most $R_0+2h+2$.

Its core and its tree paths avoid $I_{i-1}$, and its interior avoids $X$. Moreover $P_i$, including its endpoints, still avoids $P^*$, so the live-leaf estimate remains valid at the next step. The used-vertex sets are updated by $Q_i=Q_{i-1}\cup V(P_i)$ and $I_i=I_{i-1}\cup\Int(S_i)$; the discarded parts of $Q$ are not used.

Thus $S_1,\ldots,S_q$ have pairwise disjoint interiors, join consecutive vertices of $C$, and have no internal vertex in $X$. Their union is a simple cycle in the required order. Every root or tree edge has an endpoint outside $X$, and every core avoids $X$. The new cycle therefore has no edge in $H[X]$, and its length is at most $q(R_0+2h+2)$.
\end{proof}

\subsection{Short families under a local edge bound}

For integers $i\ge1$, define $M_i=2048\cdot256^{i-1}$ and $A_i=6\cdot3200^{i-1}$. For a fixed integer $j\ge1$ and a real ambient scale $N\ge15$, let
\begin{equation}\label{eq:localfunctions}
a_j=20j+40,\qquad
\tau_j(N)=(\log N)^{a_j},\qquad
P_i(N)=A_i(\log N)^{3i-2}\quad(1\le i\le j).
\end{equation}
The quantity $\tau_j(N)$ is the cutoff for local vertex-set sizes, and $P_i(N)$ is the proposed outer-cycle length bound. In this subsection, asymptotic estimates are taken as $N$ tends to infinity, with $j$ fixed.

\begin{lemma}\label{lem:localshort}
For every fixed $j\ge1$, there is $N_j\ge15$ such that the following holds for $N\ge N_j$, $\lambda\ge1$ and $1\le i\le j$. Let $F$ be a graph with at most $N$ vertices. If $e(F)>M_i\lambda|F|$ and, for every $W\subseteq V(F)$,
\begin{equation}\label{eq:LS}
e(F[W])\le\lambda|W|\qquad(0<|W|\le\tau_j(N)),
\end{equation}
then $F$ contains $i$ nested cycles without geometric crossings, with outermost cycle of length at most $P_i(N)$. We may also require $P_i(N)\le\tau_j(N)$ for every $i\le j$ and $N\ge N_j$.
\end{lemma}
\begin{proof}
First suppose that $j\ge2$. For $L=\log N$, take $c=12j+24$, $h=\lceil c\log_2L\rceil$, $s=2^h$, $p=\lfloor h/2\rfloor$, and $Q=A_{j-1}L^{3j-5}$. Choose $N_j$ sufficiently large that, for every $N\ge N_j$ and every $q\le Q$, we have $h\ge2$, $2h+2\le R(N)$, $104qs\le\tau_j(N)<N$, and~\eqref{eq:B1}--\eqref{eq:B2} with $R_0=R(N)$.

These requirements are compatible. Indeed, $s=\Theta_j(L^c)$, $2^p=\Theta_j(L^{c/2})$ and $R(N)=\Theta(L^3)$. The relevant exponent differences are
\begin{align*}
a_j-((3j-5)+c)&=5j+21,\\
c/2-(3j-5+3)&=3j+14,\\
c-((3j-5)+c/2)&=3j+17.
\end{align*}
They are all positive. After division by $s$, the three terms on the left of~\eqref{eq:B2} are respectively $O_j(L^{3j-5-c})$, $O_j(L^{3j-5-c/2})$, and $O_j(L^{3j-2-c})$, each of which is $o((\log L)^{-2})$. Increase $N_j$ also so that $P_i(N)\le\tau_j(N)$ for all $i\le j$. When $j=1$, only this last condition and the short-cycle estimate below are needed.

We induct on $i$, using the same $N$ at every step. For $i=1$, Lemmas~\ref{lem:prune} and~\ref{lem:shortcycle} give a cycle of length at most $2\log_2|F|+3\le6\log N=P_1(N)$.

Let $i\ge2$. By Lemma~\ref{lem:extract}, $F$ has a weak expander $H$ with $e(H)>M_i\lambda|H|/2$ and $\delta(H)>M_i\lambda/2\ge256$. Apply Lemma~\ref{lem:reservoir} to obtain $A,B$. Since $M_i=256M_{i-1}$, $e(H[A])\ge e(H)/128>M_{i-1}\lambda|A|$. Condition~\eqref{eq:LS} is inherited by $H[A]$. The induction hypothesis, at the same ambient scale $N$, gives $i-1$ layers with outermost cycle $C$ of length $q\le P_{i-1}(N)\le Q$.

Every vertex of $A\cup B$ has more than $M_i\lambda/16$ neighbours in $B$. With $m=4qs$, the bound $26m\le\tau_j(N)$ and $M_i\ge1024$ give
\[
e(F[W])\le\lambda|W|\le(M_i\lambda/1024)|W|
\qquad(0<|W|\le26m).
\]
Lemma~\ref{lem:trees}, with ambient graph $G=F$ and parameter $M=M_i\lambda$, supplies the private trees. Lemma~\ref{lem:wrap} gives the next outer cycle, whose length is at most
\[
q(R(N)+2h+2)\le3200q(\log N)^3\le P_i(N).
\]
Here $R(N)\le1600(\log N)^3$ for $N\ge15$. The new edges have an endpoint outside $V(C)$, whereas every old cycle lies in $F[V(C)]$. Thus edge-disjointness, containment and cyclic order all hold.

Finally, the density hypothesis and~\eqref{eq:LS} force $|F|>\tau_j(N)$, since otherwise~\eqref{eq:LS} could be applied to $V(F)$. The same observation applies to each recursive call, so no lower-order exception is needed.
\end{proof}

\begin{corollary}\label{cor:reservedlocal}
Fix $j\ge1$ and $r\ge32$, and let $\lambda=8r$ and $K=128M_j\lambda$. Suppose that $F\subseteq H$, $|F|\le N$, $N\ge N_j$, condition~\eqref{eq:LS} holds, and $e(F)>K|F|$. Then $F$ contains a $j$-layer family with outermost cycle of length at most $P_j(N)$, and there is a set $B\subseteq V(F)$ disjoint from its outer vertex set $X$ such that $d_H(x,B)\ge2r$ for every $x\in X$ and $\delta(H[B])\ge r$.
\end{corollary}
\begin{proof}
By Lemma~\ref{lem:prune}, there is a nonempty $J\subseteq F$ with $\delta(J)\ge K$ and $e(J)>K|J|$. Lemma~\ref{lem:reservoir} gives disjoint $A,B$ such that $e(J[A])\ge e(J)/128>M_j\lambda|A|$. Apply Lemma~\ref{lem:localshort} to $J[A]$ at scale $N$. The same reservoir $B$ has degree at least $K/4\ge2r$ from every chosen root and minimum degree at least $K/8\ge r$. These bounds also hold in $H$.
\end{proof}

\section{Choosing the inner cycles}\label{sec:selection}

The local construction requires an edge bound on small vertex sets. In this section we obtain two such bounds by choosing the inner cycles together with a reservoir. The first is enough to make the chosen family short. The second rules out small cuts that could prevent a vertex of the family from having two external paths.

\subsection{Keeping an external reservoir}

The next lemma is a version of Lemma~\ref{lem:reservoir} in which a prescribed external subgraph is retained. Its edge bound is in terms of $e(H[Q])$, even if vertices of $Q$ have many neighbours outside $Q$.

\begin{lemma}\label{lem:anchored}
Let $r>0$, let $V(H)=Q\dcup Z$, and suppose that $\delta(H)\ge D\ge\max\{256,8r\}$ and $\delta(H[Z])\ge r$. Then there are disjoint $A,B\subseteq Q$ such that $e(H[A])\ge e(H[Q])/128$ and
\[
d_H(x,Z\cup B)\ge d_H(x)/4\quad(x\in A),\qquad
\delta(H[Z\cup B])\ge r.
\]
\end{lemma}
\begin{proof}
Include each vertex of $Q$ independently in $A_0$ with probability $1/8$. Write $\mu(U)=\sum_{u\in U}d_H(u)$, with all degrees measured in $H$, and let $B^{\times}=\{v\in Q:d_H(v,A_0)>d_H(v)/4\}$. The random variable $d_H(v,A_0)$ is dominated by $\operatorname{Bin}(d_H(v),1/8)$, so the estimate in Lemma~\ref{lem:reservoir} gives \mbox{$\Pr(v\in B^{\times})\le e^{-d_H(v)/24}$}. Also this event is impossible unless $4d_{H[Q]}(v)>d_H(v)$. Consequently,
\[
\begin{aligned}
\mathbb E\mu(B^{\times})
&\le e^{-D/24}\sum_{\substack{v\in Q\\4d_{H[Q]}(v)>d_H(v)}}d_H(v)\\
&\le4e^{-D/24}\sum_{v\in Q}d_{H[Q]}(v)
=8e(H[Q])e^{-D/24}.
\end{aligned}
\]

Let $B_0=Q\setminus A_0$. Repeatedly delete a vertex of $B_0$ whose current degree into $Z$ together with the remaining part of $B_0$ is less than $d_H(v)/8$. No vertex of $Z$ is deleted. Let $L$ be the deleted set and $B=B_0\setminus L$.

A deleted vertex outside $B^{\times}$ initially has at least $3d_H(v)/4$ neighbours in $Z\cup B_0$. At deletion, more than $5d_H(v)/8$ of its neighbours must have been deleted earlier. Counting by deletion order as before gives
\[
\frac58\mu(L\setminus B^{\times})\le e(H[L])\le\frac12\mu(L),
\qquad \mu(L)\le5\mu(B^{\times}).
\]
Remove also from $A_0$ the vertices with fewer than $d_H(v)/4$ neighbours in $Z\cup B$, and denote the remaining set by $A$. A removed vertex outside $B^{\times}$ has more than $d_H(v)/2$ neighbours in $L$. The total degree of vertices removed from $A_0$ is therefore at most $\mu(B^{\times})+2\mu(L)\le11\mu(B^{\times})$. Together with the bound on $\mathbb E\mu(B^{\times})$, this yields
\[
\mathbb E e(H[A])\ge e(H[Q])\left(\frac1{64}-88e^{-D/24}\right)
\ge e(H[Q])/128,
\]
where the last inequality holds for $D\ge256$. Choose an outcome attaining this bound.

The required degrees from $A$ follow from its definition. Every vertex of $B$ has at least $d_H(v)/8\ge r$ neighbours in $Z\cup B$, and every vertex of $Z$ retains all its original neighbours in $Z$. Thus \mbox{$\delta(H[Z\cup B])\ge r$}.
\end{proof}

\subsection{A minimum family with a reservoir}

Fix an integer $j\ge1$. Throughout this section, assume that there is a constant $\kappa_j>0$ such that
\begin{equation}\label{eq:indhyp}
e(F)>\kappa_j|F|
\quad\Longrightarrow\quad
F\text{ contains }j\text{ nested cycles without geometric crossings}.
\end{equation}
This is the induction hypothesis used in the proof of Theorem~\ref{thm:main}. Choose
\begin{equation}\label{eq:rK}
r\in\mathbb N,\qquad
r\ge1024\max\{1,\kappa_j\},\qquad
\lambda=8r,\qquad K=128M_j\lambda.
\end{equation}
Here $\kappa_j$ is the preceding forcing density, $r$ the reservoir degree threshold, and $\lambda,K$ the density thresholds used below.
In a fixed host $H$, call a $j$-layer family \emph{eligible} if its outer vertex set $X$ admits a disjoint set $B$ with $d_H(x,B)\ge2r$ for every $x\in X$ and $\delta(H[B])\ge r$. We call the vertices of $X$ the \emph{roots} and $B$ a \emph{reservoir} for the family. The reservoir may change when the family is replaced, but the thresholds $2r,r$ remain fixed.

\begin{lemma}\label{lem:denseeligible}
Every nonempty $S\subseteq V(H)$ with $e(H[S])>\lambda|S|$ contains an eligible $j$-layer family whose reservoir also lies in $S$.
\end{lemma}
\begin{proof}
Apply Lemma~\ref{lem:prune} at threshold $8r$, obtaining a nonempty $J\subseteq H[S]$ with $\delta(J)\ge8r$ and $e(J)>8r|J|$. By Lemma~\ref{lem:reservoir}, there are disjoint $A,B$ such that
\[
e(J[A])\ge e(J)/128>\frac r{16}|J|
\ge\frac r{16}|A|>\kappa_j|A|.
\]
Use~\eqref{eq:indhyp} in $J[A]$. The degree bounds into $B$ are at least $2r$ at every root and at least $r$ within $B$, so the resulting family is eligible.
\end{proof}

Suppose now that $\delta(H)\ge D\ge128r$, with $D$ as fixed in~\eqref{eq:D} below. Eligible families exist. Indeed, Lemma~\ref{lem:reservoir} applied directly to $H$ gives a core $A$ satisfying $e(H[A])\ge D|H|/256>\kappa_j|A|$, and a reservoir with the required degree bounds. Choose an eligible family whose outermost cycle has as few vertices as possible. Let $X$ be its outer vertex set and $q=|X|$. The minimum is over all eligible families and their possible reservoirs in $H$. Lemma~\ref{lem:denseeligible} gives
\begin{equation}\label{eq:MIN}
e(H[S])\le\lambda|S|\qquad(0<|S|<q).
\end{equation}

\Needspace{5\baselineskip}
\begin{lemma}\label{lem:peelX}
Starting from $H-X$, repeatedly delete a vertex of current degree less than $r$. Let $E\subseteq V(H)\setminus X$ be the set of additionally deleted vertices and let $R=V(H)\setminus(X\cup E)$ be the final remaining set. Thus $V(H)=X\dcup E\dcup R$, and
\begin{equation}\label{eq:CORE}
|E|<q,\qquad
d_H(x,R)\ge2r\quad(x\in X),\qquad
\delta(H[R])\ge r.
\end{equation}
\end{lemma}
\begin{proof}
A reservoir witnessing the eligibility of $X$ survives throughout the process. Otherwise its first deleted vertex would still have at least $r$ neighbours in that reservoir. The two degree conclusions follow.

Suppose that $|E|\ge q$. If $q\le3$, then $\delta(H-X)\ge D-q>r$, so no first deletion is possible. We may therefore assume $q\ge4$. Let $P$ consist of the first $q$ deleted vertices. At deletion, each has more than $D-r$ neighbours in $X$ or among earlier deleted vertices. Each edge inside $P$ is counted only at its later endpoint, and each edge between $P$ and $X$ is counted once. Hence $e(H[X\cup P])>(D-r)q$ and $|X\cup P|=2q$. A uniformly chosen $(q-1)$-element subset $S$ of $X\cup P$ has expected edge density
\[
\frac{\mathbb E e(H[S])}{q-1}
=\frac{e(H[X\cup P])(q-2)}{(2q)(2q-1)}
>\frac{D-r}{2}\frac{q-2}{2q-1}
\ge\frac{D-r}{8}>8r.
\]
Here $(q-2)/(2q-1)\ge1/4$ for $q\ge4$. Some such $S$ contradicts~\eqref{eq:MIN}.
\end{proof}

\begin{lemma}\label{lem:shorteligible}
If $D>2K$ and $n=|H|\ge N_j$, then $q\le\tau_j(n)$.
\end{lemma}
\begin{proof}
If $q>\tau_j(n)$, inequality~\eqref{eq:MIN} gives the local condition~\eqref{eq:LS} throughout $H$ at scale $N=n$. Since $e(H)\ge Dn/2>Kn$, Corollary~\ref{cor:reservedlocal} gives an eligible family with outermost cycle of length at most $P_j(n)\le\tau_j(n)<q$, a contradiction.
\end{proof}

The next bound has no factor depending on $r$. This stronger estimate is needed for the smaller cuts in the path argument.

\begin{lemma}\label{lem:smallsets}
Every nonempty $S\subseteq V(H)$ with $2|S|<q$ satisfies
\begin{equation}\label{eq:smallMIN}
e(H[S])\le256\kappa_j|S|.
\end{equation}
\end{lemma}
\begin{proof}
Let $s=|S|$, and start a new deletion process from the whole graph $H-S$, again at degree threshold $r$. If $s$ additional vertices were deleted, the first $s$ of them would form a set $P$ with $e(H[S\cup P])>(D-r)s>2\lambda s$. Since $|S\cup P|=2s<q$, this contradicts~\eqref{eq:MIN}. Thus the deleted set $L$ has size less than $s$. Let $Q=S\cup L$ and $Z=V(H)\setminus Q$. Then $|Q|<2s<q$ and $\delta(H[Z])\ge r$. In particular, $Z$ is nonempty because $|Q|<q\le|H|$.

Suppose that~\eqref{eq:smallMIN} fails. Lemma~\ref{lem:anchored} gives $A,B\subseteq Q$ with $A\cap B=\varnothing$ and
\[
e(H[A])\ge e(H[Q])/128
\ge e(H[S])/128>2\kappa_js>\kappa_j|A|.
\]
By~\eqref{eq:indhyp}, $H[A]$ contains a $j$-layer family. Its vertices have at least $D/4\ge2r$ neighbours in $Z\cup B$, and $\delta(H[Z\cup B])\ge r$. It is therefore eligible. Its outermost cycle, however, has at most $|A|\le|Q|<q$ vertices, contradicting the definition of $q$.
\end{proof}

The last replacement may change every member of the old family. It is valid because the minimum was taken over all eligible $j$-layer families in $H$, rather than over extensions of a prescribed family.

\subsection{Two paths from every vertex}

For real $s\ge1$, define
\[
F_\eps(s)=\frac{64s}{\eps}\log^2\!\left(\frac{15s}{\eps}\right),
\qquad
\mathcal N(q)=\lceil F_\eps(4q)\rceil+4q\quad(q\in\mathbb N).
\]
The scalar functions $F_\eps(s)$ and $\mathcal N(q)$ bound set and subgraph orders, not neighbourhoods. For $U\subseteq V(H)$ and real $s\ge1$, a weak expander $H$ satisfies
\begin{equation}\label{eq:INV}
|U|\le|H|/2,\quad |N_H(U)|\le s
\quad\Longrightarrow\quad |U|<F_\eps(s).
\end{equation}
To see this for nonempty $U$, note that $x/\log^2(15x)$ is increasing for $x\ge1$. Let $z=s/\eps$, $\ell=\log(15z)$ and $x_0=64z\ell^2$. Since $\ell\ge\log15$, $\log(15x_0)=\ell+\log64+2\log\ell\le4\ell$. Indeed, $3\ell-\log64-2\log\ell$ is positive at $\ell=\log15$ and increasing thereafter. Thus $\frac{x_0}{\log^2(15x_0)}\ge4z>z$, whereas expansion gives $|U|/\log^2(15|U|)\le z$. Monotonicity proves~\eqref{eq:INV}; the empty-set case is immediate.

Choose an integer $q_0\ge3$ so that for every integer $q>q_0$,
\begin{equation}\label{eq:q0}
\mathcal N(q)\ge N_j,\qquad
\tau_j(\mathcal N(q))<q,\qquad
P_j(\mathcal N(q))<q.
\end{equation}
Such a choice is possible because $\mathcal N(q)=O(q\log^2q)$ and each fixed power of $\log q$ is $o(q)$. We now fix an integer
\begin{equation}\label{eq:D}
D\ge\max\{128r,128K,\mathcal N(q_0)+2\}.
\end{equation}
The constants are chosen in the order $r,\lambda,K,N_j,q_0,D$. None of them depends on the order of $H$.

\begin{lemma}\label{lem:fullfan}
Assume~\eqref{eq:indhyp} and the choices above. Let $H$ be a weak expander on $n$ vertices with $\delta(H)\ge D$, and let $X,E,R$ come from a minimum eligible family and Lemma~\ref{lem:peelX}; write $q=|X|$. Every set $T\subseteq R$ with
\begin{equation}\label{eq:target}
|T|>F_\eps(4q)+2q
\end{equation}
admits two $X$--$T$ paths from each vertex of $X$, such that all $2q$ paths lie in $H[X\cup R]$, are disjoint outside $X$, have distinct endpoints in $T$, and have no other vertex of $X$ in their interiors.
\end{lemma}
\begin{proof}
Construct a directed network with source $\sigma$ and sink $\omega$; an arc $uv$ is directed from $u$ to $v$. Each arc $\sigma x$, for $x\in X$, has capacity two. Replace each $v\in R$ by two distinct network vertices, its in-copy $v^-$ and out-copy $v^+$, with an arc $v^-v^+$ of capacity one. Include arcs $xv^-$ for each edge $xv$ from $X$ to $R$, both arcs $u^+v^-,v^+u^-$ for each edge $uv$ of $H[R]$, and arcs $t^+\omega$ for $t\in T$. Give all these remaining arcs capacity $2q+1$. An integral flow assigns integer arc values between zero and capacity, with conservation except at $\sigma,\omega$; its value is the net outflow from $\sigma$. A cut is a partition separating $\sigma$ from $\omega$; its capacity sums the capacities of arcs directed from the source side to the sink side.

An integral flow of value $2q$ gives the required paths after discarding flow cycles and, if necessary, stopping each path at its first target vertex. Suppose that no such flow exists, and take a minimum cut of capacity less than $2q$. No arc of capacity $2q+1$ crosses the cut. We may also assume that whenever $v^+$ is on the source side, $v^-$ is on that side: moving $v^-$ there cannot increase the cut capacity, since its only outgoing arc is the split arc to $v^+$.

Let $U\subseteq X$ be the roots on the source side, $B\subseteq R$ the vertices with both copies there, and $Z\subseteq R$ the vertices with only their in-copy there. The cut capacity is $2(q-|U|)+|Z|<2q$. Hence
\begin{equation}\label{eq:cut}
\begin{gathered}
U\ne\varnothing,\qquad |Z|<2|U|,\\
\Gamma_H(U)\cap R\subseteq B\cup Z,\qquad
N_{H[R]}(B)\subseteq Z,\qquad B\cap T=\varnothing.
\end{gathered}
\end{equation}
The neighbourhood assertions follow because no large-capacity arc crosses the cut.

In $H$, the external neighbourhood of $B$ is contained in $X\cup E\cup Z$, which has fewer than $4q$ vertices. If $|B|>n/2$, then $B'=R\setminus(B\cup Z)$ has fewer than $n/2$ vertices. Since $N_{H[R]}(B)\subseteq Z$, no vertex of $B'$ is adjacent to $B$, so $N_H(B')\subseteq X\cup E\cup Z$ as well. Moreover, $B'$ contains $T\setminus Z$, whose order exceeds $F_\eps(4q)$ by~\eqref{eq:target}. This contradicts~\eqref{eq:INV}. Therefore $|B|\le n/2$, and~\eqref{eq:INV} gives $|B|<F_\eps(4q)$. Let $t=|U|+|B|$.

\medskip
\noindent\emph{Case 1: $t<q/8$.}
Let $S=U\cup B\cup Z$. Then $0<|S|<3t<3q/8$, so $2|S|<q$. By~\eqref{eq:CORE} and~\eqref{eq:cut}, at least $2r|U|$ edges join $U$ to $B\cup Z$. Also, the edges inside $H[B\cup Z]$ incident with $B$ number at least $r|B|/2$. These edge collections are disjoint. Since $|Z|<2|U|$,
\[
e(H[S])\ge2r|U|+\frac r2|B|
>\frac r2|S|\ge512\kappa_j|S|,
\]
contrary to Lemma~\ref{lem:smallsets}.

\medskip
\noindent\emph{Case 2: $t\ge q/8$.}
Let $W=X\cup E\cup B\cup Z$. All neighbours in $H$ of every vertex of $U\cup B$ lie in $W$. The original minimum degree therefore gives
\[
2e(H[W])\ge Dt,\qquad
|W|<2q+2t,\qquad |W|<\mathcal N(q).
\]
The last bound uses $|B|<F_\eps(4q)$, $|E|<q$ and $|Z|<2q$.

If $q\le q_0$, a vertex of the nonempty set $U$ would have all its neighbours in a set of fewer than $\mathcal N(q_0)<D$ vertices, a contradiction. Thus $q>q_0$. Moreover, the preceding bounds give
\[
\frac{e(H[W])}{|W|}>\frac{Dt}{4(q+t)}\ge\frac D{36}>K.
\]
At ambient scale $N=\mathcal N(q)$, conditions~\eqref{eq:q0} and~\eqref{eq:MIN} give every local edge bound required by Corollary~\ref{cor:reservedlocal}. That corollary produces an eligible $j$-layer family in $H[W]$ whose outermost cycle has length at most $P_j(\mathcal N(q))<q$, again contradicting minimality.

Both cases are impossible. The maximum-flow min-cut theorem yields a flow of value $2q$. The resulting paths lie in $H[X\cup R]$, since the network contains no vertex of $E$.
\end{proof}

\section{Proof of Theorem~\ref{thm:main}}\label{sec:proof}

It remains to join consecutive vertices of the chosen inner cycle. A common target set alone does not prescribe these pairings. We first find a large clique minor in the reservoir and then apply Lemma~\ref{lem:rooted}.

\Needspace{5\baselineskip}
\begin{lemma}\label{lem:largeclique}
For every fixed $a>0$, there is $n_a$ such that the following holds. Let $H$ be a weak expander on $n\ge n_a$ vertices, and suppose that $V(H)=X\dcup E\dcup R$, $1\le q=|X|\le(\log n)^a$, and $|E|<q$. If
\begin{equation}\label{eq:h}
h=\lceil F_\eps(4q)\rceil+4q+2,
\end{equation}
then $H[R]$ contains a $K_h$ model.
\end{lemma}
\begin{proof}
Uniformly for $q\le(\log n)^a$, the quantity $h$ is bounded by a fixed power of $\log n$ times a fixed power of $\log\log n$. Thus $n_a$ can be chosen so that, for all such $q$ and $n\ge n_a$,
\begin{equation}\label{eq:sepbudget}
h^{3/2}\sqrt n+2q\le\frac{\eps n}{16\log^2(15n)},
\qquad 2q\le n/4.
\end{equation}

Suppose that $H[R]$ has no $K_h$ minor, and let $m=|R|>n-2q\ge3n/4$. Lemma~\ref{lem:AST} gives a set $Z\subseteq R$ of size $z\le h^{3/2}\sqrt m$ such that $R\setminus Z$ is the union of two mutually nonadjacent sets, each of size at most $2m/3$. Let $U$ be the smaller of these sets. Then $m/3-z\le |U|\le m/2\le n/2$. By~\eqref{eq:sepbudget}, $z\le n/8$, so $|U|\ge n/8$. Its external neighbourhood in $H$ is contained in $Z\cup X\cup E$, and therefore
\[
|N_H(U)|<z+2q\le\frac{\eps n}{16\log^2(15n)}.
\]
But expansion gives
\[
|N_H(U)|\ge\frac{\eps|U|}{\log^2(15|U|)}
\ge\frac{\eps n}{8\log^2(15n)},
\]
a contradiction.
\end{proof}

\begin{lemma}\label{lem:orderedfromfans}
Let $X=\{x_1,\ldots,x_q\}$ and $R$ be disjoint vertex sets of $H$, where $q\ge3$, and prescribe the cyclic order $x_1,\ldots,x_q$. Suppose that $H[R]$ contains a $K_h$ model with $h$ as in~\eqref{eq:h}. Suppose also that every $T\subseteq R$ satisfying~\eqref{eq:target} admits two paths from each root in $X$ to distinct vertices of $T$, such that the $2q$ paths have their remaining vertices in $R$ and are disjoint outside $X$. Then there are paths $Q_i$ from $x_i$ to $x_{i+1}$, with indices modulo $q$, whose interiors are in $R$ and pairwise vertex-disjoint. Their union is a simple cycle in the prescribed order and has no edge with both endpoints in $X$.
\end{lemma}
\begin{proof}
Form a graph $J$ from $H[R]$ by adding two new vertices $x_i^-,x_i^+$ for each $x_i$, called its \emph{clones}. All these vertices are distinct; each has neighbourhood $N_H(x_i)\cap R$, and they are mutually nonadjacent. Let $Y=\{x_i^-,x_i^+:1\le i\le q\}$. The clique model in $H[R]$ remains a model in $J-Y$. For every target $T$ under consideration, the path hypothesis gives $2q$ vertex-disjoint $Y$--$T$ paths in $J$, one starting at each clone.

Suppose that $J$ has a separation $(A,B)$ of order less than $2q$ such that $Y\subseteq A$ and a whole branch set of the model is in $B\setminus A$. At most $2q-1$ branch sets meet $A\cap B$. Every branch set avoiding this separator is connected and hence lies entirely on one open side. As all branch sets are pairwise adjacent and one is entirely in $B\setminus A$, every branch set avoiding the separator must lie in $B\setminus A$.

Consequently, $T=(B\setminus A)\cap R$ has at least $h-(2q-1)>F_\eps(4q)+2q$ vertices. The corresponding $2q$ vertex-disjoint $Y$--$T$ paths all meet $A\cap B$, which has fewer than $2q$ vertices. This is impossible.

Since $h\ge4q=2|Y|$, Lemma~\ref{lem:rooted} now gives a $K_{2q}$ model rooted at $Y$. For each $i$, use the branch sets containing $x_i^+$ and $x_{i+1}^-$ and one edge between them to obtain a path from $x_i^+$ to $x_{i+1}^-$. All $q$ paths are vertex-disjoint, because each rooted branch set is used exactly once. No terminal is an internal vertex of a path: every branch set contains exactly its own terminal.

Identify $x_i^-$ and $x_i^+$ with $x_i$ for each $i$. The resulting paths $Q_i$ have pairwise disjoint interiors in $R$, connect consecutive roots, and contain no other vertex of $X$. Since the clone set was independent, every path has length at least two. Their union is therefore a simple cycle in the prescribed order, and every edge has an endpoint in $R$.
\end{proof}

\begin{proof}[Proof of Theorem~\ref{thm:main}]
We prove by induction on $j$ that there is a constant $\kappa_j$ for which~\eqref{eq:indhyp} holds. For $j=1$, take $\kappa_1=1$, since a forest on $n$ vertices has at most $n-1$ edges.

Assume the assertion for a fixed $j\ge1$. Choose $r,\lambda,K$ as in~\eqref{eq:rK}. Lemma~\ref{lem:localshort} supplies $N_j$ and the functions $\tau_j,P_i$ in~\eqref{eq:localfunctions}. Choose $q_0$ to satisfy~\eqref{eq:q0}, and then choose $D$ as in~\eqref{eq:D}. Finally, let $n_0\ge\max\{N_j,n_{a_j}\}$, where $n_{a_j}$ is the threshold in Lemma~\ref{lem:largeclique} with $a=a_j$. All these choices depend only on $j$ and the already chosen $\kappa_j$.

Consider a weak expander $H$ with $|H|\ge n_0$ and $\delta(H)\ge D$. Choose a minimum eligible $j$-layer family, and let $X$ be its outer vertex set. Lemmas~\ref{lem:peelX} and~\ref{lem:shorteligible} give a partition $X,E,R$ with $|E|<q=|X|\le\tau_j(|H|)=(\log|H|)^{a_j}$. Lemma~\ref{lem:fullfan} supplies two disjoint paths per root to every sufficiently large target set in $R$. Lemma~\ref{lem:largeclique} provides the required clique model in $H[R]$, and Lemma~\ref{lem:orderedfromfans} gives an additional outer cycle.

Every old cycle has all its vertices in $X$, whereas each edge of the new cycle has an endpoint in $R$. Thus the new cycle is edge-disjoint from every old cycle. It contains $X$ in the cyclic order of the old outer cycle, and the other containment and order conditions follow from the chosen inner family. Hence $H$ contains $j+1$ layers.

For an arbitrary graph $G$, define $\kappa_{j+1}=2\max\{D,n_0\}$. If $e(G)>\kappa_{j+1}|G|$, Lemma~\ref{lem:extract} gives a weak expander $H\subseteq G$ with $\delta(H)>\kappa_{j+1}/2\ge\max\{D,n_0\}$. In particular $|H|>n_0$, so the preceding argument applies. This proves~\eqref{eq:indhyp} at level $j+1$ for every graph order and completes the induction. Taking $C_k=\kappa_k$ proves the stated forcing result. For $k\ge2$ and $n\ge1$,
\[
f_k(n)\le\lfloor\kappa_kn\rfloor+1\le(\kappa_k+1)n.
\]
\end{proof}

These recursively defined constants are large, but their size does not depend on the order of the input graph. Only the selected inner family is required to be short. The paths obtained from the rooted clique model, and hence the additional outer cycle, may have unrestricted length. This is sufficient because the next inductive step uses only existence and makes its own choice of an eligible inner family.

\smallskip
\noindent\textbf{Use of AI tools.}
Building on our earlier work~\cite{AGH}, discussions with ChatGPT (GPT-6 Astra Pro) helped us develop the idea of reselecting
the inner cycles while retaining an external reservoir, and of using
rooted clique minors to obtain the required cyclic order without requiring the new outer cycle to be short. The tool also assisted
with language polishing. The authors take full responsibility for
the mathematical content and final manuscript.

\end{document}